\documentclass[a4paper]{amsart}

\usepackage{amsthm}
\usepackage{amsmath}
\usepackage{amssymb}
\usepackage[dvipdfmx]{graphicx}
\usepackage{caption}

\makeatletter
\def\mojiparline#1{
    \newcounter{mpl}
    \setcounter{mpl}{#1}
    \@tempdima=\linewidth
    \advance\@tempdima by-\value{mpl}zw
    \addtocounter{mpl}{-1}
    \divide\@tempdima by \value{mpl}
    \advance\kanjiskip by\@tempdima
    \advance\parindent by\@tempdima
}

\makeatother

\makeatother

\title[Dehn twist presentations of hyperelliptic periodic diffeomorphisms]
{Dehn twist presentations of hyperelliptic periodic diffeomorphisms on closed surfaces}
\author{Norihisa Takahashi}
\author{Hiraku Nozawa}
\date{}

\address{Norihisa Takahashi, Graduate School of Science and Engineering, Ritsumeikan University, Nojihigashi 1-1-1, Kusatsu, Shiga, 525-8577, Japan}
\email{ec042070@ed.ritsumei.ac.jp}

\address{Hiraku Nozawa, Department of Mathematical Sciences, Colleges of Science and Engineering, Ritsumeikan University, Nojihigashi 1-1-1, Kusatsu, Shiga, 525-8577, Japan}
\thanks{The second author is supported by JSPS Grant-in-Aid for Scientific Research 17K14195}
\email{hnozawa@fc.ritsumei.ac.jp}

\newtheorem{thm}{Theorem}
\newtheorem{cor}[thm]{Cororally}
\newtheorem{prop}[thm]{Proposition}
\newtheorem*{thm:}{Main Theorem}
\newtheorem{lem}[thm]{Lemma}
\theoremstyle{definition}

\theoremstyle{remark}
\newtheorem{rem}[thm]{Remark}

\newcommand{\opn}{\operatorname}

\usepackage{color}
\definecolor{darkgreen}{cmyk}{1,0,1,.2}
\definecolor{darkorchid}{rgb}{0.6, 0.2, 0.8}
\definecolor{persimmon}{rgb}{0.93, 0.35, 0.0}

\newdimen\theight
\def\TeXref#1{%
             \leavevmode\vadjust{\setbox0=\hbox{{\tt
                     \quad\quad  {\small \textrm #1}}}%
             \theight=\ht0
             \advance\theight by \lineskip
             \kern -\theight \vbox to
             \theight{\rightline{\rlap{\box0}}%
             \vss}%
             }}%

\begin{document}

\begin{abstract}
	We classify up to conjugacy the group generated by a commuting pair of a periodic diffeomorphism and a hyperelliptic involution on an oriented closed surface.
	This result can be viewed as a refinement of Ishizaka's result on classification of the mapping classes of hyperelliptic periodic diffeomorphisms.
	As an application, we obtain the Dehn twist presentations of hyperelliptic periodic mapping classes, which are closely related to the ones obtained by Ishizaka.
\end{abstract}

\maketitle

\section{Introduction}

	Periodic diffeomorphisms on closed surfaces have been studied being motivated by the geometry of the mapping class group, Lefschetz fibrations and degenerating families of Riemann surfaces.
	They are related to periodic mapping classes by so-called Nielsen realization theorem \cite{Kerckhoff,Nielsen}:
	Any finite group of the mapping class group of an oriented closed surface $\Sigma_{g}$ of genus $g > 1$ is lifted to an isomorphic finite subgroup of a hyperbolic metric of the same period.
	Periodic diffeomorphisms and periodic mapping classes on $\Sigma_{g}$ are classified up to conjugacy in various cases by using so-called the total valency (see e.g.\ \cite{Ashikaga-Ishizaka,Hirose,Hirose2,Hirose-Kasahara,Ishizaka2}).
	Finite group actions on $\Sigma_g$ have been studied from various viewpoints (see e.g. \cite{Broughton,Nakamura-Nakanishi}).

	One of main problems is to determine periodic mapping classes which can be presented as a product of right-handed Dehn twists (see e.g.\ \cite{Hirose,BH71}).
	Ishizaka \cite{Ishizaka3,Ishizaka} classified hyperelliptic periodic diffeomorphisms and give their right-handed Dehn twist presentations based on the resolution of singularities of families of Riemann surfaces.
	Recall that an involution $I$ on $\Sigma_{g}$ is called hyperelliptic if it fixes $2g+2$ points (see Figure \ref{fig:hypinv}).
	Such involution is unique up to conjugacy, and has the maximum number of fixed points among involutions on $\Sigma_{g}$.
	A diffeomorphism on $\Sigma_{g}$ is called hyperelliptic if it commutes with a hyperelliptic involution $I$.
	The subgroup of $\opn{Mod}(\Sigma_{g})$ generated by the mapping classes of all such diffeomorphisms is a particular case of symmetric mapping class groups in the Birman-Hilden theory \cite{BH73}.
	
		\begin{figure}[ht]
			\centering
			\captionsetup{width=0.8\linewidth}
			\includegraphics[scale=0.6 ]{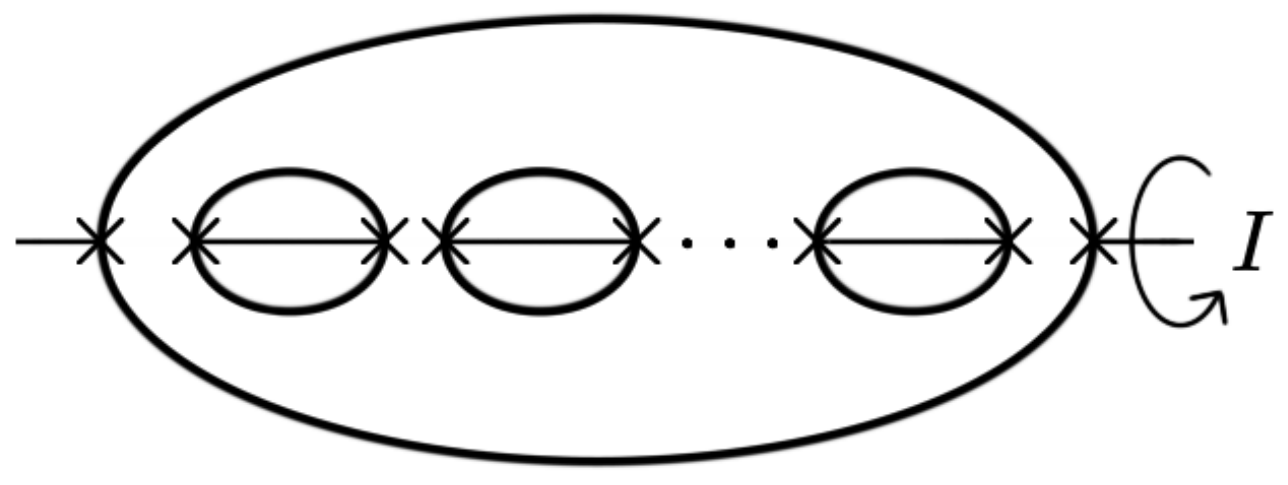}
			\caption{Hyperelliptic involution}
			\label{fig:hypinv}
		\end{figure}
	
	Let us state the main result, which can be viewed as a refinement of Ishizaka's classification of conjugacy classes of hyperelliptic periodic diffeomorphisms (\cite[Lemma 1.3]{Ishizaka}).
	We will apply the main result to obtain Dehn twist presentations of hyperelliptic periodic diffeomorphisms, which are closely related to the ones obtained by Ishizaka (\cite[Theorem 2.1]{Ishizaka}).
	In order to state the result, let us present three examples of hyperelliptic periodic diffeomorphisms $f_1$, $f_2$ and $f_3$.
	Consider an oriented closed surface $\Sigma_{g}$ of genus $g$ obtained by identifying pairs of the edges of the regular $(8g+4)$-gon in Figures \ref{fig:phi1} as follows:
	Identify the pairs of edges obtained from the pair marked with $\alpha$ by the clockwise $\frac{2k\pi}{4g+2}$-rotation for $k=0,1, \ldots, 4g+1$.
	We consider the standard orientation on $\Sigma_{g}$.
	Now $f_{1}$ is defined to be the periodic diffeomorphism induced by the clockwise $\frac{2\pi}{4g+2}$-rotation on the $(8g+4)$-gon.
	We define $f_{2}$ and $f_{3}$ in a way similar to $f_{1}$ by taking a regular $8g$-gon in Figure \ref{fig:phi2} and a regular $(4g+4)$-gon in Figure \ref{fig:phi3} respectively in place of the $(8g+4)$-gon.
	A hyperelliptic involution which commutes with $f_{i}$ is given by $I=f_{i}^{n/2}$, where $n$ is the period of $f_{i}$ for $i=1$ and $2$.
	Finally, $f_{3}$ is the periodic diffeomorphism induced from the clockwise $\frac{2\pi}{2g+2}$-rotation on a regular $(4g+4)$-gon in Figure \ref{fig:phi3}. 
	A hyperelliptic involution which commutes with $f_{3}$ is a $\pi$-rotation of each fundamental domain of $f_{3}$ as given in Figure \ref{fig:phi3}.
	Let $G_{i} = \langle f_{i}, I \rangle$ be the subgroup of $\opn{Diff}_{+}(\Sigma_{g})$ generated by $f_{i}$ and $I$ for $i=1$, $2$ and $3$.
	The main result of this article is as follows.
	
		\begin{thm}
		\label{thm:vis}
			With the above notations, for any periodic diffeomorphism $f$ on $\Sigma_{g}$ which commutes with a hyperelliptic involution $I$, the subgroup $G=\langle f, I \rangle$ of $\opn{Diff}_{+}(\Sigma_{g})$ generated by $f$ and $I$ is conjugate to a subgroup of either of $G_1$, $G_2$ or $G_3$.
		    \end{thm}

	The conjugacy class of $f$ in this theorem was originally classified by Ishizaka's theorem based on resolutions of singularities of families of Riemann surfaces.
	Our proof of Theorem \ref{thm:vis} is elementary.

		\begin{figure}[ht]
			\centering
			\begin{tabular}{cc}
			\begin{minipage}{0.45\textwidth}
				\centering
				\captionsetup{width=0.8\linewidth}
				\includegraphics[width=0.9\textwidth ]{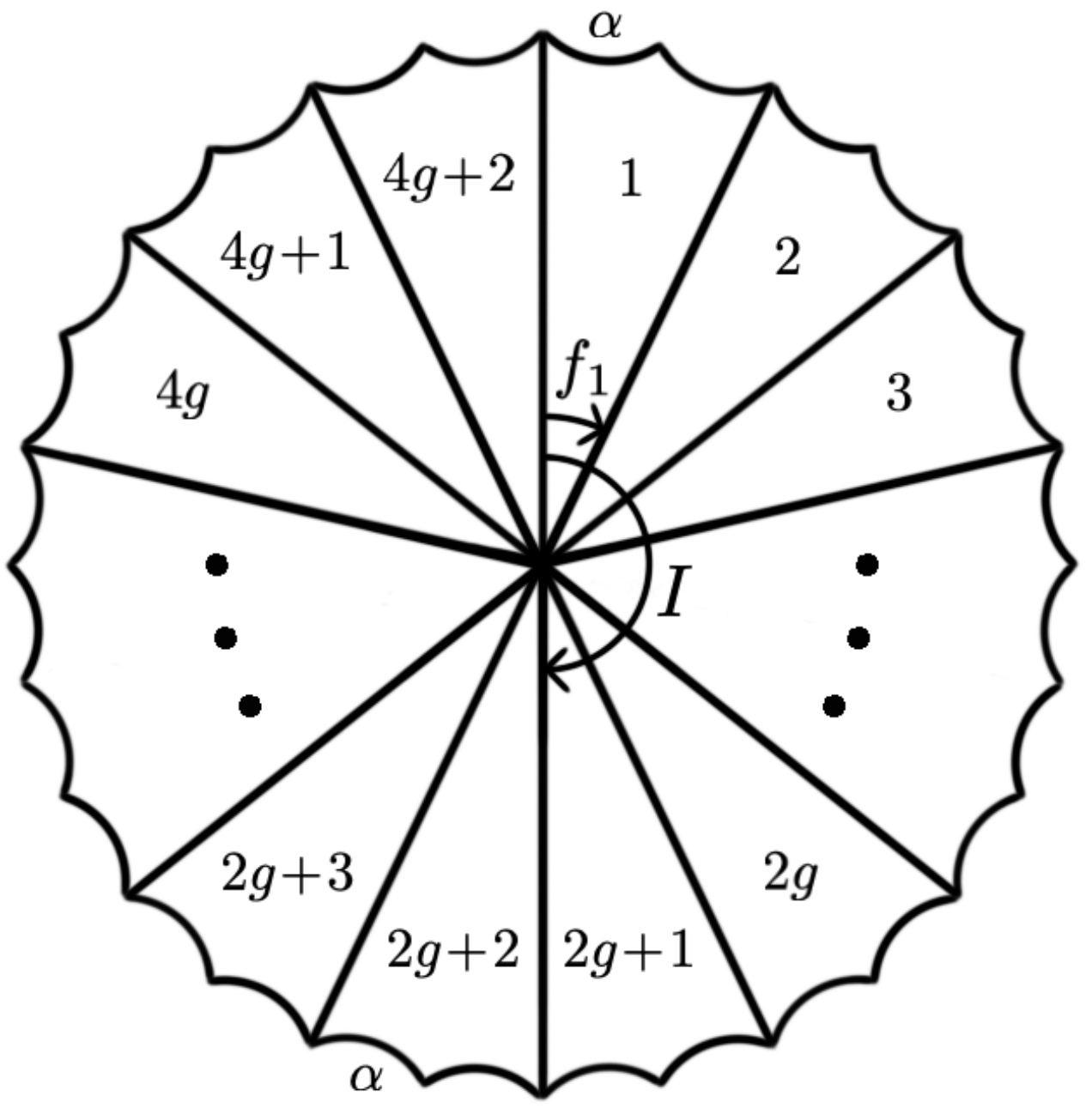}
				\caption{$G_1$-action on $\Sigma_{g}$: $f_{1}$ is the $\frac{2\pi}{4g+2}$-rotation and $I$ is the $\pi$-rotation}
				\label{fig:phi1}
			\end{minipage}\hfill
			\begin{minipage}{0.45\textwidth}
				\centering
				\captionsetup{width=0.8\linewidth}
				\includegraphics[width=0.9\textwidth ]{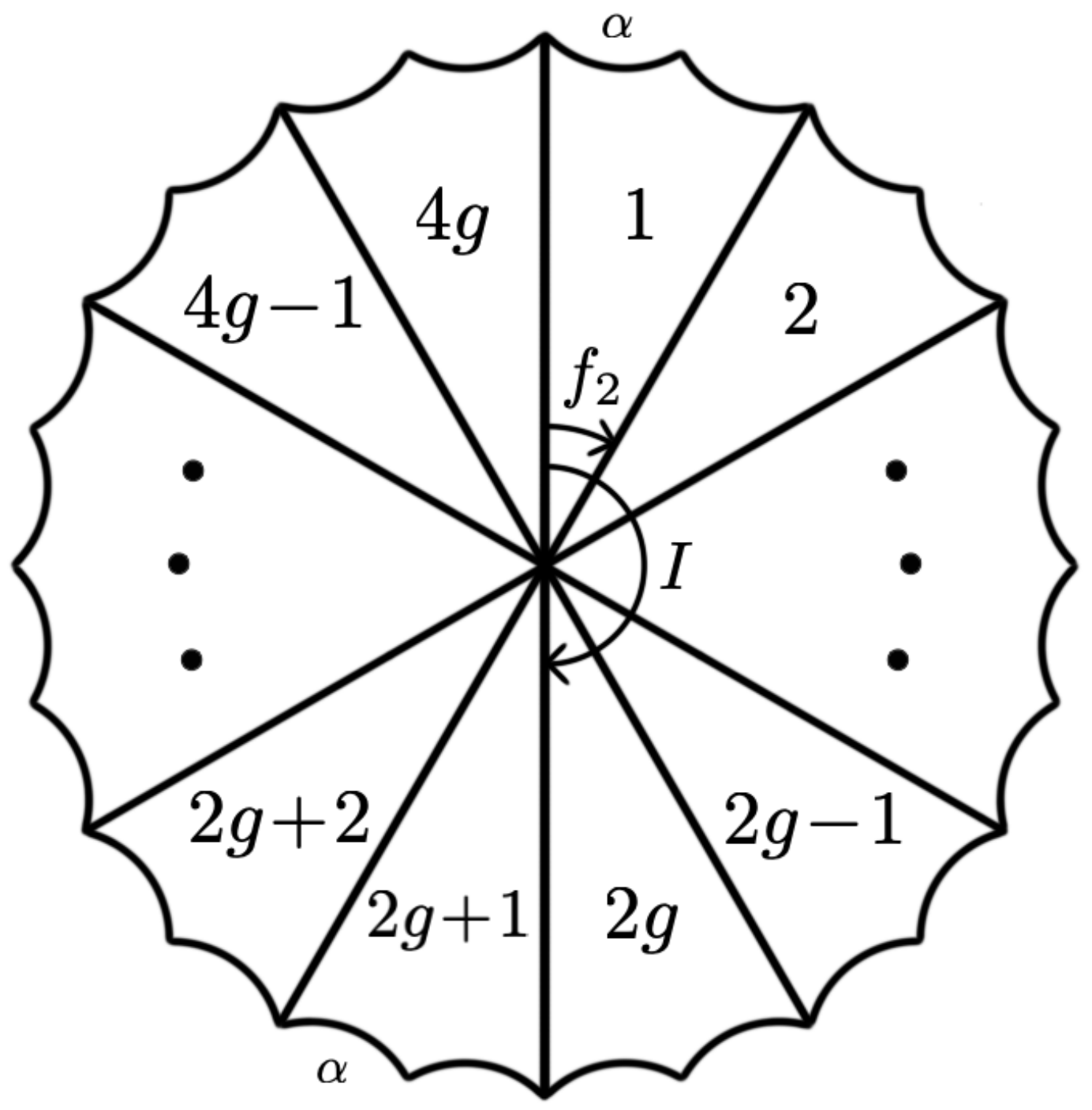}
				\caption{$G_2$-action on $\Sigma_{g}$: $f_{1}$ is the $\frac{2\pi}{4g}$-rotation and $I$ is the $\pi$-rotation
				}
				\label{fig:phi2}
			\end{minipage}\\
			\begin{minipage}{0.45\textwidth}
				\centering
				\captionsetup{width=0.9\linewidth}
				\includegraphics[width=0.9\textwidth ]{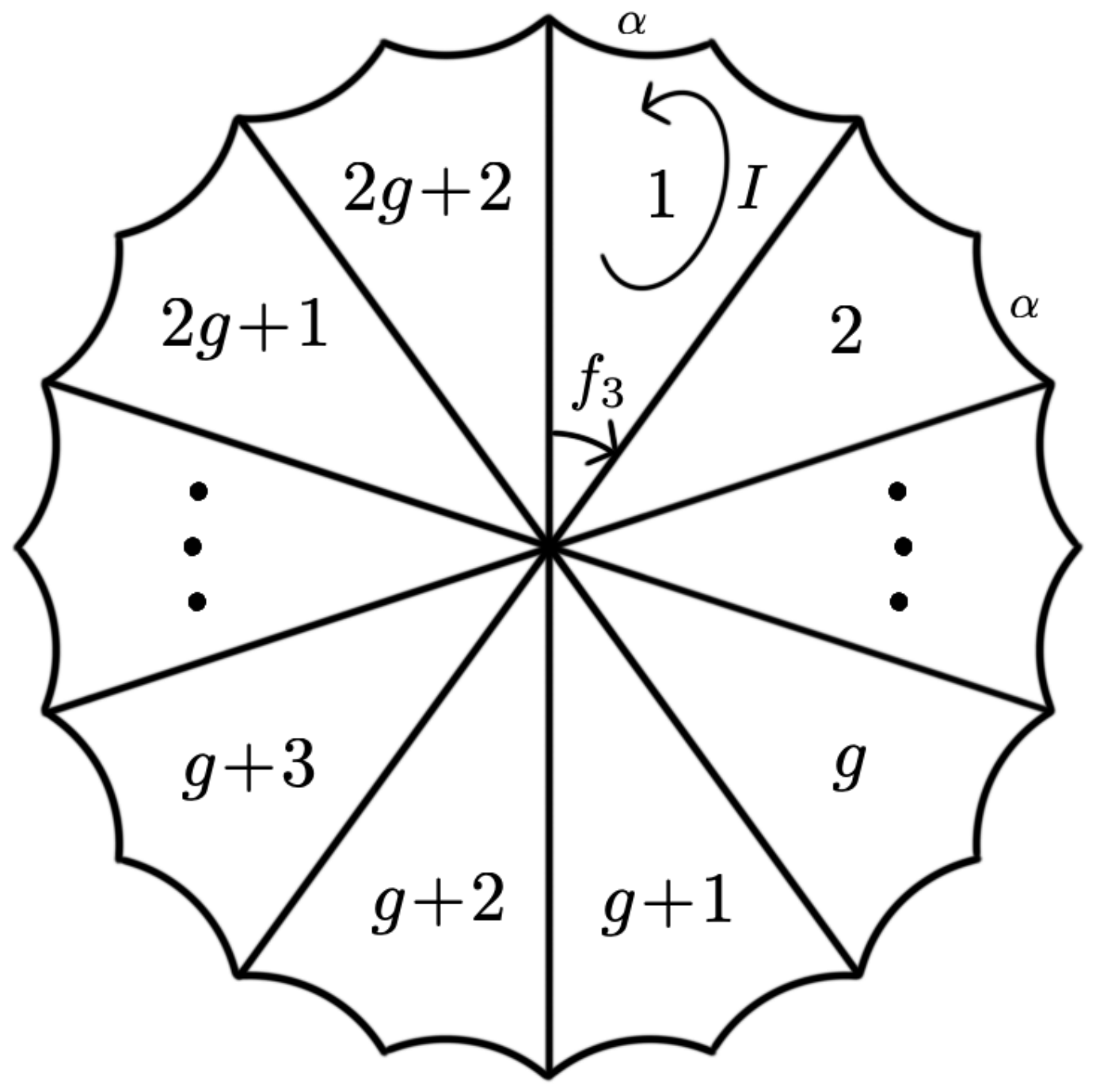}
				\caption{$G_3$-action on $\Sigma_{g}$: \\
				$f_{3}$ is the $\frac{2\pi}{2g+2}$-rotation and $I$ is an involution which rotates each fundamental domain by $\pi$}
				\label{fig:phi3}
			\end{minipage}\hfill
			\begin{minipage}{0.5\textwidth}
				\centering
	\captionsetup{width=0.9\linewidth}
				\includegraphics[scale=0.45]{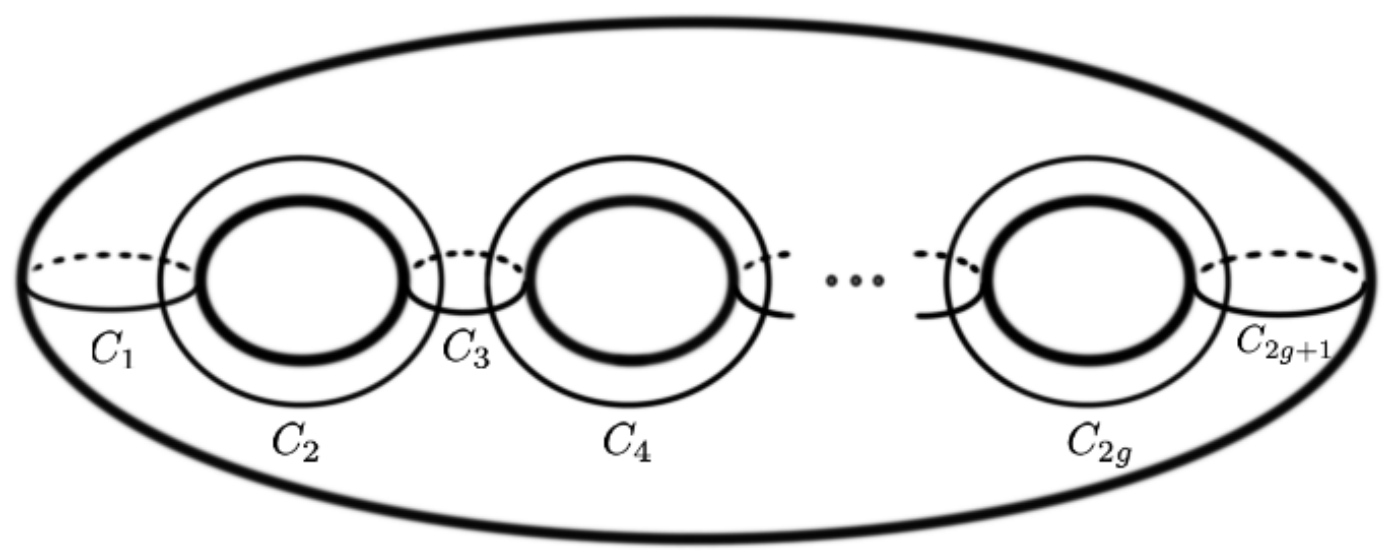}
				\caption{Humphries generators}
				\label{fig:humgen}
			\end{minipage}
		\end{tabular}
		\end{figure}

	Let us consider Dehn twist presentations of periodic hyperelliptic mapping classes.
	Here recall so-called Humphries generators of the mapping class group $\operatorname{Mod}(\Sigma_{g})$, which are Dehn twists along $2g+1$ curves $C_{1}, \ldots, C_{2g+1}$ shown in Figure \ref{fig:humgen}.
	We denote by $[\phi]$ the conjugacy class of a mapping class $\phi$.
	For a simple closed curve $C$ on $\Sigma_{g}$, we denote by the same symbol $C$ the Dehn twist along the curve $C$.
	Ishizaka obtained the following right-handed Dehn twist presentations of periodic hyperelliptic mapping classes.

		\begin{thm}[{\cite[Theorem 2.1]{Ishizaka}}]
		\label{thm:Ishizaka}
			The conjugacy class of the mapping class of any hyperelliptic periodic diffeomorphism $\phi$ on $\Sigma_{g}$ is equal to either of $[\phi_{1}^k]$, $[\phi_{2}^k]$, $[\phi_{3}^{k}]$ or $[I\phi_{3}^{2k}]$ for some positive integer $k$, where $I$ is the hyperelliptic involution, and we have
				\begin{enumerate}
				\renewcommand{\labelenumi}{(\roman{enumi})}
					\item $[\phi_1]= [C_1 C_2 \cdots C_{2g}]$,
					\item $[\phi_2]= [C_1 C_2 \cdots C_{2g} C_{2g}]$,
					\item $[\phi_3]= [C_1 C_2 \cdots C_{2g} C_{2g+1}]$,
				\end{enumerate}
where $C_{i}$ are Humphries generators (Figure \ref{fig:humgen}).
		\end{thm}

	In \cite{Ishizaka}, the mapping classes $\phi_{1}$, $\phi_{2}$ and $\phi_{3}$ are determined by the total valencies (see Theorem \ref{thm:Ishizaka1}).

	By applying Theorem \ref{thm:vis}, we obtain the following Dehn twist presentations of hyperelliptic periodic diffeomorphisms with an elementary argument.
	Let $x \in \Sigma_{g}$ be the barycenter of the polygon in each of Figures \ref{fig:phi1}, \ref{fig:phi2} and \ref{fig:phi3}, which is a fixed point of $f_{i}$ for each $i=1$, $2$ and $3$, respectively.
	$\opn{Mod}(\Sigma_{g}, x)$ is the pointed mapping class group of $\Sigma_{g}$ with marked point $x$.
	In Section \ref{sec:3}, we obtain the following Dehn twist presentations of $f_1$, $f_{2}$ and $f_{3}$ in $\opn{Mod}(\Sigma_{g}, x)$.
	
		\begin{thm}\label{thm:Dehn}
			Let $A_i$, $B_i$ and $D_i$ be simple closed curves shown in Figures \ref{fig:phi1-2}, \ref{fig:phi2-2} and \ref{fig:phi3-2}, respectively.
			Then, we have
			\begin{enumerate}
			\renewcommand{\labelenumi}{(\roman{enumi})}
				\item $(f_1) = A_{1} A_{2} \cdots A_{2g}$,
				\item $(f_2) = B_{1} B_{2g} B_{2g-1} \cdots B_{1}$ and
				\item $(f_3) = D_{1} D_{2} \cdots D_{2g+1}$,
			\end{enumerate}
			where $(f_i)$ is the pointed mapping class of $f_i$ for $i=1,2$ and $3$.
		\end{thm}
		
 Our Dehn twist presentations in $\opn{Mod}(\Sigma_{g},x)$ is related to Ishizaka's Dehn twist presentations in $\opn{Mod}(\Sigma_{g})$ via the forgetful map $\opn{Mod}(\Sigma_{g},x) \to \opn{Mod}(\Sigma_{g})$.

		\begin{figure}[ht]
			\centering
			\begin{tabular}{ccc}
			\begin{minipage}{0.33\textwidth}
				\centering
			\captionsetup{width=0.8\linewidth}	\includegraphics[width=1\textwidth ]{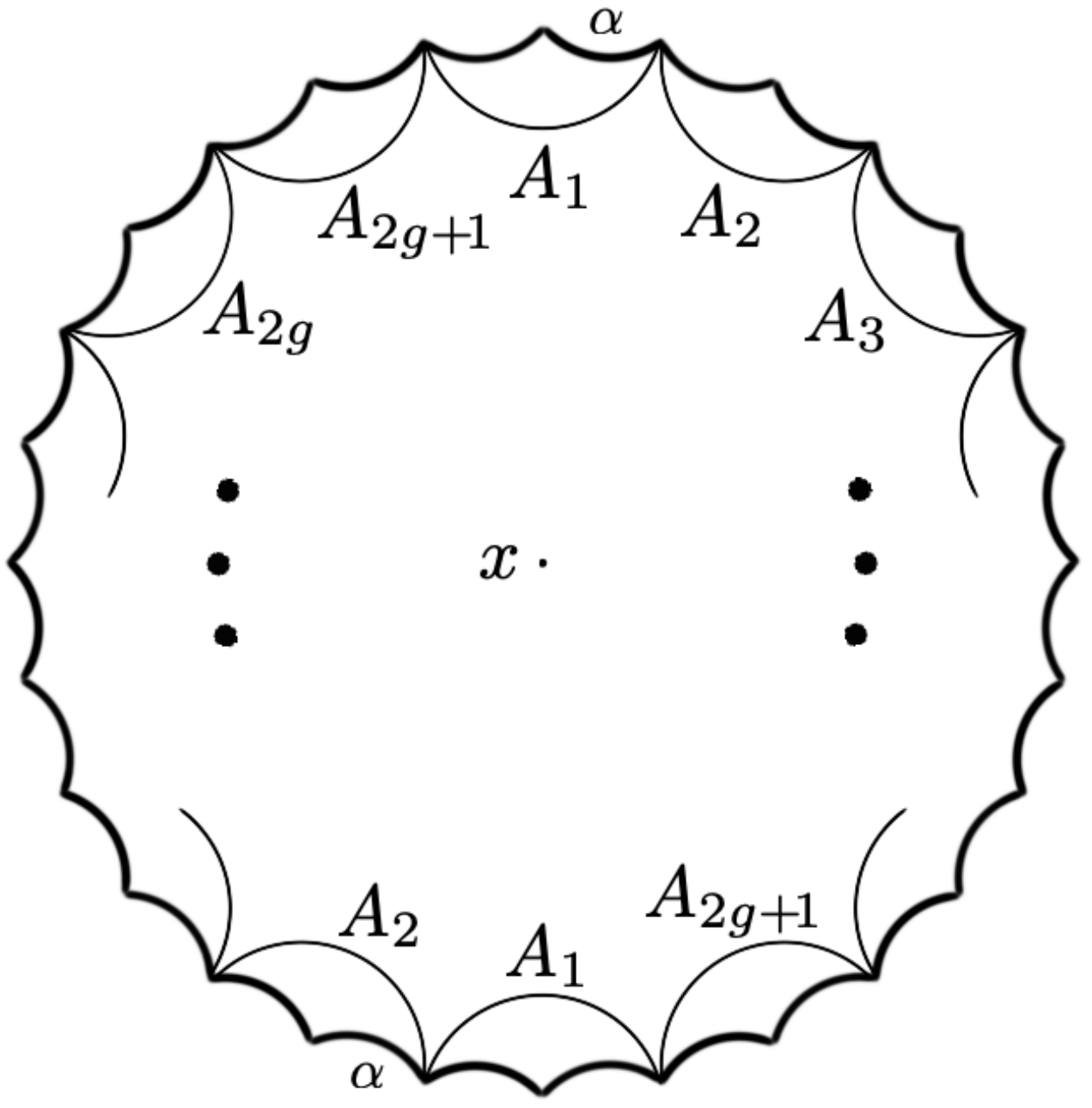}
				\caption{Curves $A_{i}$ on $\Sigma_{g}$ in Fig.\ \ref{fig:phi1}}
				\label{fig:phi1-2}
			\end{minipage}\hfill
			\begin{minipage}{0.33\textwidth}
				\centering
			\captionsetup{width=0.8\linewidth}	\includegraphics[width=1\textwidth ]{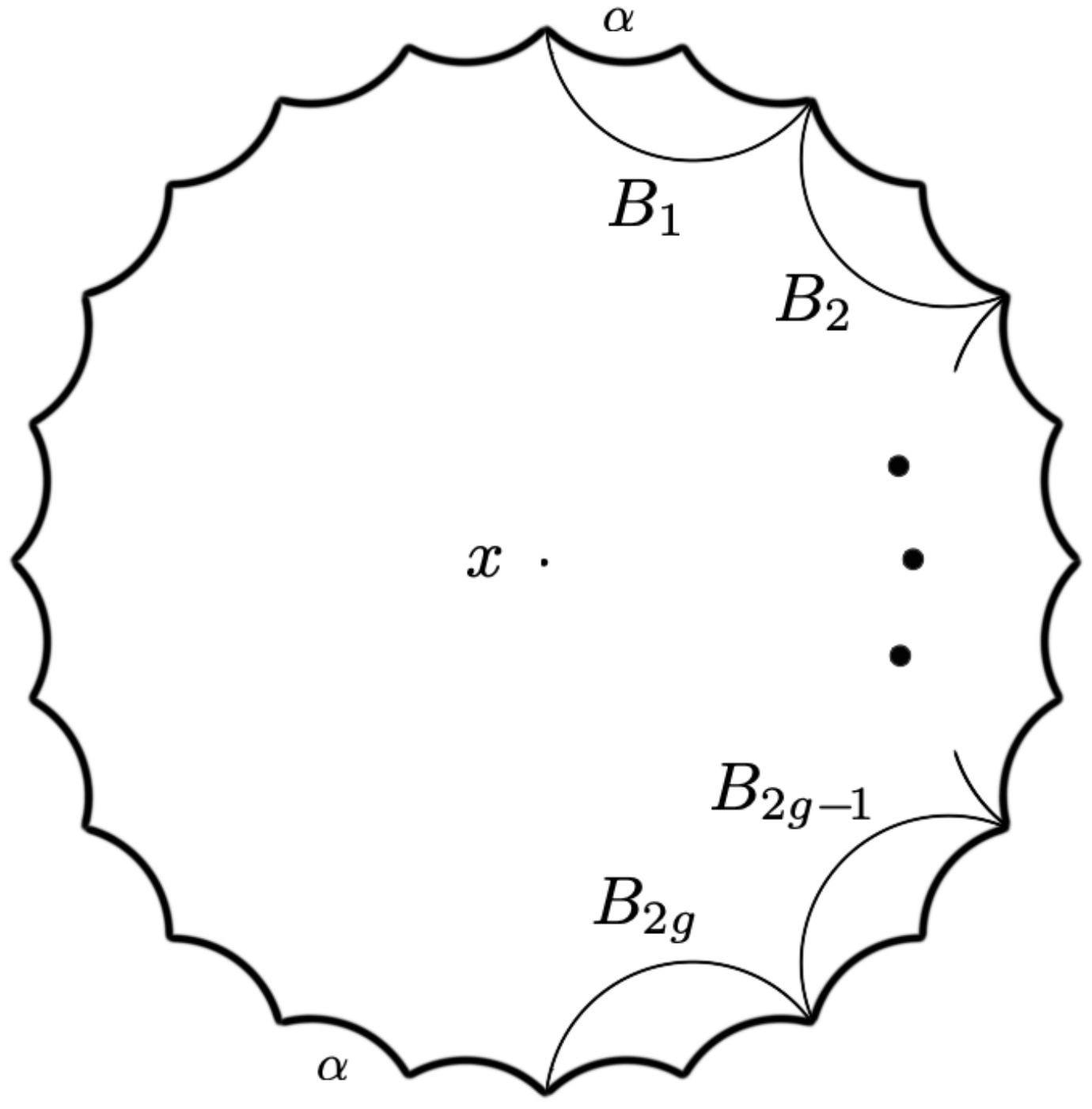}
				\caption{Curves $B_{i}$ on $\Sigma_{g}$ in Fig.\ \ref{fig:phi2}}
				\label{fig:phi2-2}
			\end{minipage}\hfill
			\begin{minipage}{0.33\textwidth}
				\centering
			\captionsetup{width=0.8\linewidth}	\includegraphics[width=1\textwidth ]{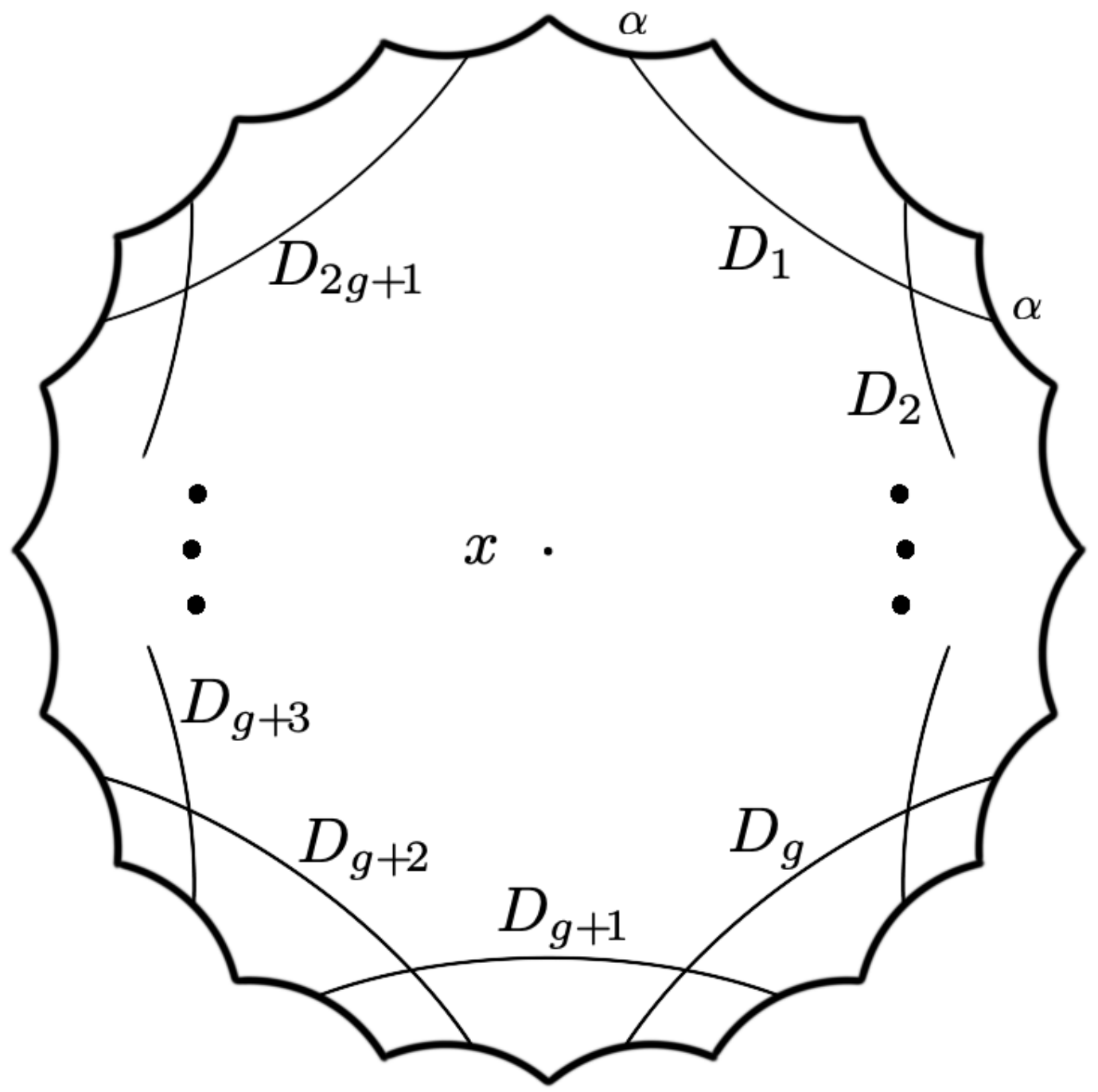}
				\caption{Curves $C_{i}$ on $\Sigma_{g}$ in Fig.\ \ref{fig:phi3}}
				\label{fig:phi3-2}
			\end{minipage}
			\end{tabular}
		\end{figure}
		
\section{Total valency of hyperelliptic periodic diffeomorphisms}
\label{sec:2}

	Let us recall the classification of the conjugacy classes of periodic diffeomorphisms on $\Sigma_{g}$ in terms of the total valency introduced in \cite{Ashikaga-Ishizaka}.
	Let $f \in \opn{Diff}_+(\Sigma_{g})$ be a periodic diffeomorphism of order $n$.
	Let $F$ be the cyclic subgroup of $\opn{Diff}_{+}(\Sigma_{g})$ generated by $f$.
	An $F$-orbit $Fx$ is called {\it multiple} if $|Fx| < |F|$, where $|\cdot|$ denote the cardinality.
	For a multiple orbit $Fx$, let $\lambda = |F_{x}|$, where $F_{x}$ is the isotropy group of $x$.
	Then, there uniquely exists $\nu \in \{1,2, \ldots, \lambda - 1\}$ such that the restriction of $f$ to a small neighborhood of $x$ is the clockwise $\frac{2\pi\nu}{\lambda}$-rotation.
	For this integer $\nu$, there uniquely exists an integer $\theta \in \{1,2, \ldots, \lambda - 1\}$ such that $\nu \theta \equiv 1 \mod \lambda$.
	This $\frac{\theta}{\lambda}$ is called the {\it valency} of $Fx$.
	Let $\frac{\theta_1}{\lambda_1}, \frac{\theta_2}{\lambda_2}, \ldots, \frac{\theta_s}{\lambda_s}$ be the valencies of all multiple orbits of $f$.
	The data $[ \, g,n \,; \frac{\theta_1}{\lambda_1} + \frac{\theta_2}{\lambda_2} + \cdots + \frac{\theta_s}{\lambda_s} \, ]$ is called the {\it total valency} of $f$. 
	By the following theorem, total valencies determine periodic diffeomorphisms on $\Sigma_{g}$ up to conjugation.
	
		 \begin{thm}[{Nielsen \cite[Section 11]{Nielsen}, see also \cite[Section 1.3]{Ashikaga-Ishizaka}, \cite[Theorem 2.1]{Hirose}}]
		\label{thm:Nielsen}
			Let $f, f'$ be periodic diffeomorphisms of $\Sigma_{g}$ with total valencies $[\,g,n \,;\, \frac{\theta_1}{\lambda_1} + \frac{\theta_2}{\lambda_2} + \cdots +\frac{\theta_s}{\lambda_s} \,]$ and $[\,g,n' \,;\,\frac{\theta_1'}{\lambda_1'} + \frac{\theta_2'}{\lambda_2'} + \cdots +\frac{\theta_{s'}'}{\lambda_{s'}'} \,]$, respectively.
			$f$ is conjugate to $f'$ if and only if the following are satisfied:
				\begin{enumerate}
				\renewcommand{\labelenumi}{(\roman{enumi})}
					\item $s=s'$,
					\item $n=n'$,
					\item after changing indices, we have $\frac{\theta_i}{\lambda_i}=\frac{\theta_i'}{\lambda_i'}$ for $i =1,2,\ldots,s$.
				\end{enumerate}
		\end{thm}
	
	By Theorem \ref{thm:Nielsen}, we identify the total valency of a periodic diffeomorphism with its conjugacy class.
	We will use the following well-known observation of Nielsen.

		\begin{prop}[{\cite[Equation (4.6)]{Nielsen}}]
		\label{prop:Nielsen}
			The sum of valencies of all multiple orbits of a periodic diffeomorphism on $\Sigma_{g}$ is an integer.
		\end{prop}

	        \begin{rem}
       		\label{rem:inv}
        			Any multiple orbit of an involution is a fixed point whose valency is $\frac{1}{2}$.
			Thus, by Proposition \ref{prop:Nielsen}, the number of fixed points of an involution is even.
	        \end{rem}

	One can define powers of the total valency of a periodic diffeomorphism $f$ by the total valency of powers of $f$.
	Note that the composite of two periodic diffeomorphisms may not be periodic, and the product of total valencies does not make sense in general.

	A mapping class $\phi$ is called hyperelliptic if it commutes with the mapping class of a hyperelliptic involution.
	A periodic hyperelliptic mapping class $\phi$ is lifted to a periodic diffeomorphism of the same order which commutes with the hyperelliptic involution by Kerckhoff's theorem \cite[Theorem 5]{Kerckhoff}.
	Ishizaka determined total valencies of hyperelliptic periodic mapping classes.

		\begin{thm}[{\cite[Lemma 1.3]{Ishizaka}}]
		\label{thm:Ishizaka1}
			Let $\phi$ be a hyperelliptic periodic mapping class of $\Sigma_{g}$ and $[\phi]$ be its conjugacy class.
			Then, there exists a positive integer $k$ such that the conjugacy class of $\phi$ is equal to one of the followings:
				\begin{enumerate}
				\renewcommand{\labelenumi}{(\roman{enumi})}
					\item $[\phi_1^k] = [\, g,4g + 2\, ; \, \frac{1}{4g+2} + \frac{g}{2g+1} + \frac{1}{2} \,] ^k$,
					\item $[\phi_2^k] = [\, g,4g\, ; \, \frac{1}{4g} + \frac{2g-1}{4g} + \frac{1}{2} \,] ^k$,
					\item $[\phi_3^k] = [\, g,2g + 2\, ; \, \frac{1}{2g+2} + \frac{1}{2g+2} + \frac{g}{g+1} \,] ^k$,
					\item $[I\phi_3^{k}]$.
				\end{enumerate}
		\end{thm}

	Let us prove Theorem \ref{thm:vis}.
	We also give an alternative proof of Theorem \ref{thm:Ishizaka1}.
	First we see the total valency of $f_i$ is equal to that of $\phi_{i}$ for $i=1, 2$ and $3$.
	
		\begin{lem}\label{lem:123}
			We have
				\begin{enumerate}
					\item $[f_1] = [\,g, 4g+2\,;\, \frac{1}{4g+2}+ \frac{g}{2g+1}+\frac{1}{2} \,]$,
					\item $[f_2] = [\,g, 4g\,;\, \frac{1}{4g}+ \frac{2g-1}{4g}+\frac{1}{2} \,]$,
					\item $[f_3] = [\,g, 2g+2\,;\, \frac{1}{2g+2}+ \frac{1}{2g+2}+\frac{g}{g+1} \,]$.
				\end{enumerate}
			Here, we denote by $[f]$ the conjugacy class of a periodic diffeomorphism $f$, which is identified with its total valency by Theorem \ref{thm:Nielsen}.
		\end{lem}
	
	\begin{proof}
		We can compute the valency of each multiple orbit of $f_{i}$ by using Figures \ref{fig:phi1}, \ref{fig:phi2} and \ref{fig:phi3}.
		There are three multiple orbits of $f_1$, which are the vertices of tetragonal fundamental domains shown in Figure \ref{fig:phi1}.
		Clearly the valency of the fixed point at the barycenter of the $(8g+4)$-gon is $\frac{1}{4g+2}$.
		One of the other multiple orbits is of period $2g+1$, and hence its valency is $1/2$.
		We can compute the valency of the multiple orbit of $f_{1}$ of period $2$ as follows:
		Take a point $y$ in the orbit, and consider a loop which goes around the boundary of a small disk centered at $y$.
		The fundamental domains appear along this loop in the order
			\[
				1, \, 4g+2, \, 2g+1, \, 2g, \, 4g+1, \ldots, 2g+4, \, 3, \, 2, \, 2g+3, \, 2g+2
			\]
		(see Figure \ref{fig:f1val2}).
		The angle made by the fundamental domains marked with $1$ and $3$ is $\frac{2g-1}{2g+1} 2\pi$.
		Here, since we have
			\[
				(2g-1)g \equiv 2g^2 - g \equiv -2g \equiv 1 \ \mod 2g+1,
			\]
		the valency of this multiple orbit is $\frac{g}{2g+1}$.
		Similarly we can compute the total valencies of $f_{2}$ and $f_{3}$. 
	\end{proof}
		
		\begin{figure}[ht]
			\centering
			\includegraphics[width=0.4\textwidth ]{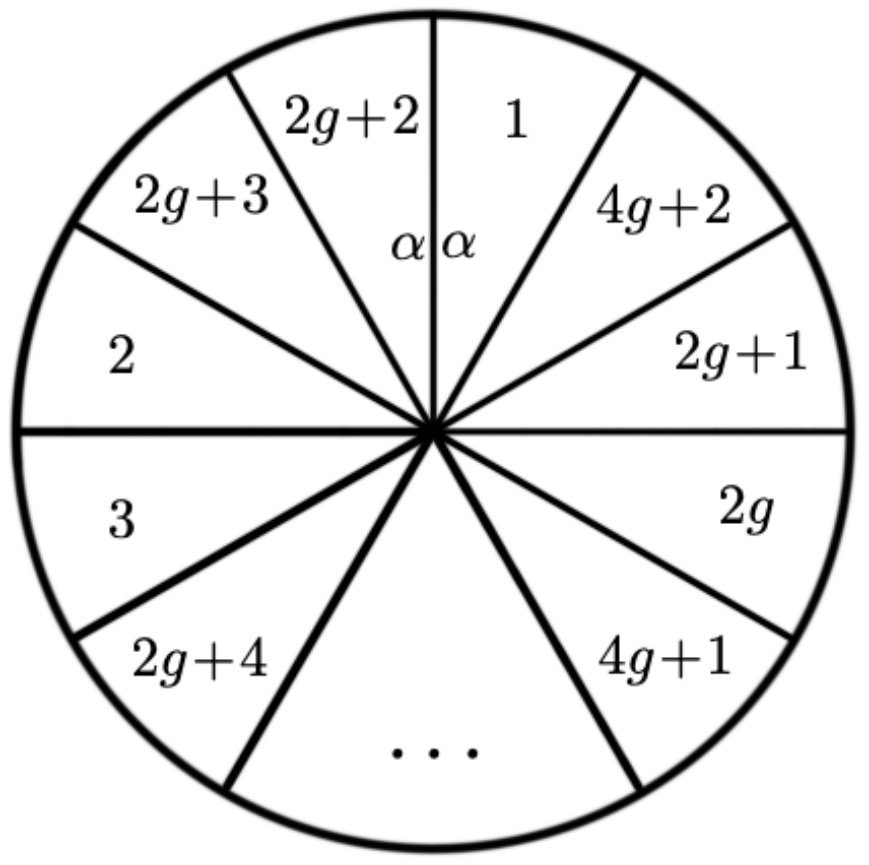}
			\caption{}
			\label{fig:f1val2}
		\end{figure}
	
		\begin{figure}[ht]
			\centering
			\begin{minipage}{0.33\textwidth}
				\centering
				\includegraphics[width=1\textwidth ]{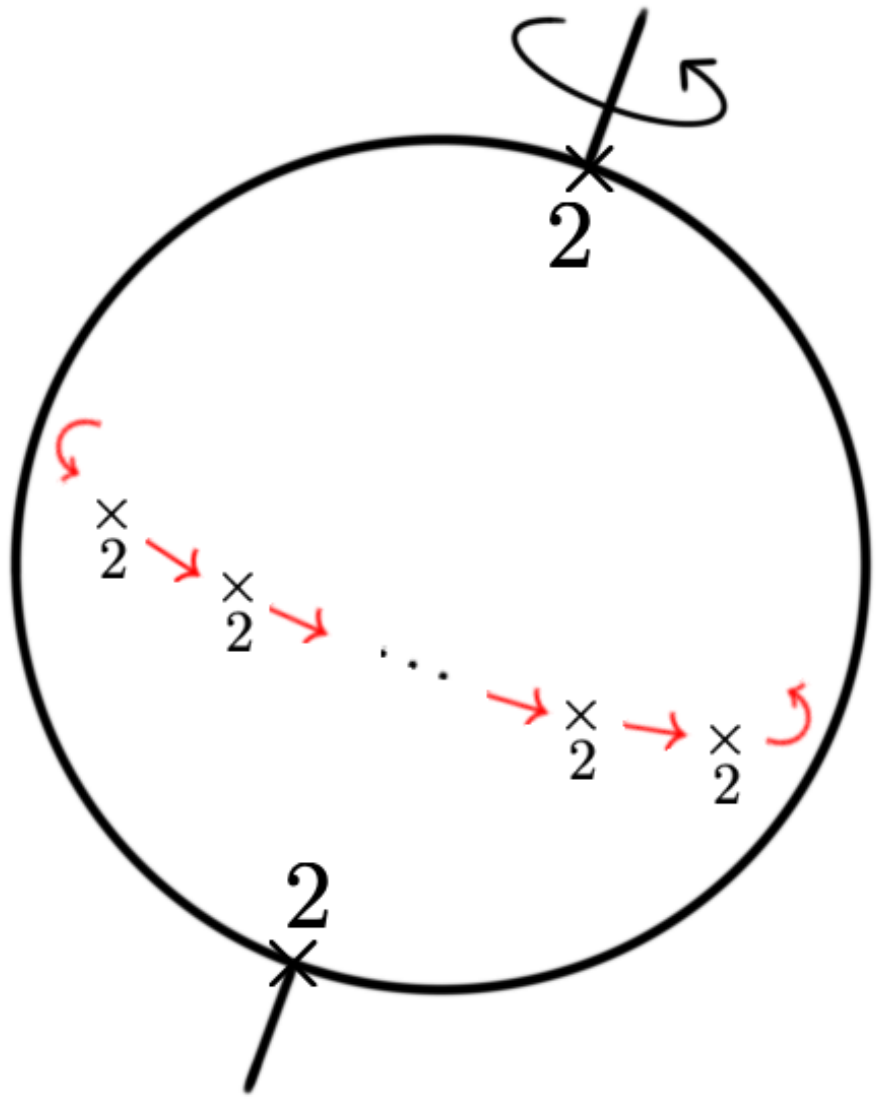}\\
				(i) The middle orbit consists of $2g$ cone points.
			\end{minipage}\hfill
			\begin{minipage}{0.33\textwidth}
				\centering
				\includegraphics[width=1\textwidth ]{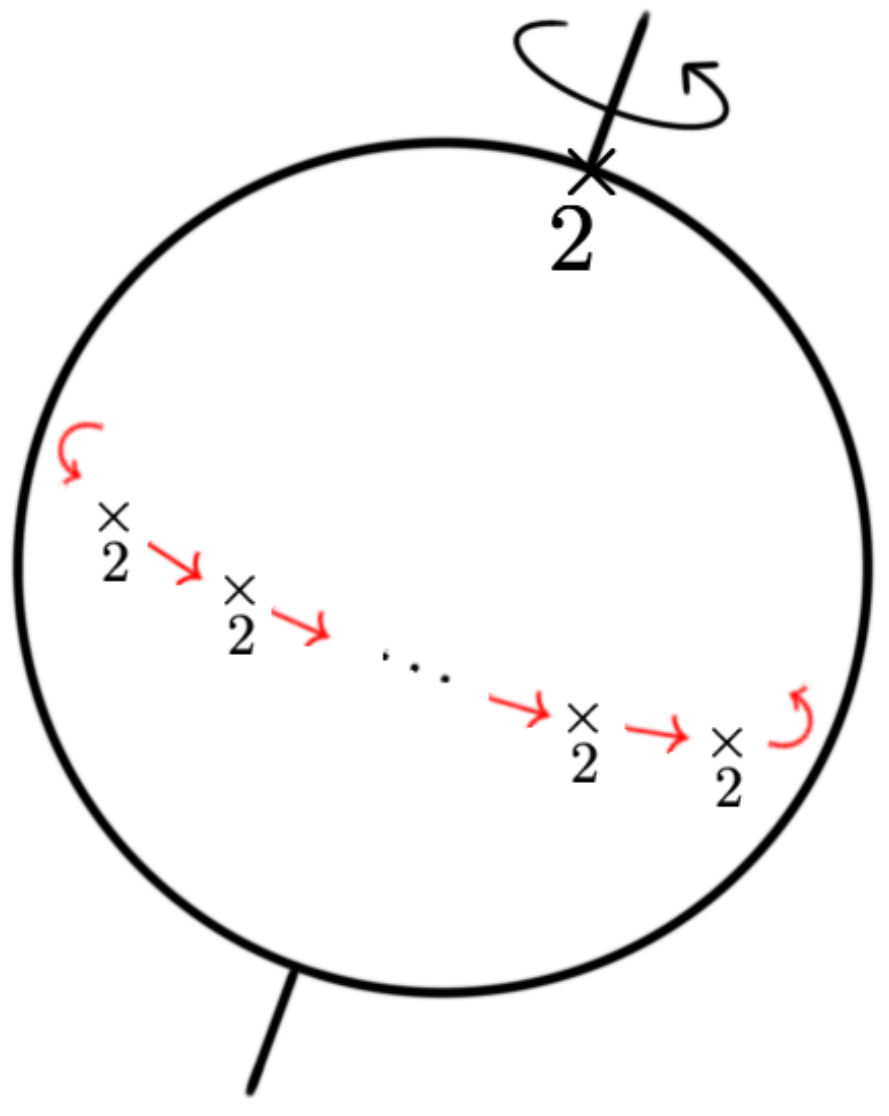}\\
				(ii) The middle orbit consists of $2g+1$ cone points.
			\end{minipage}\hfill
			\begin{minipage}{0.33\textwidth}
				\centering
				\includegraphics[width=1\textwidth ]{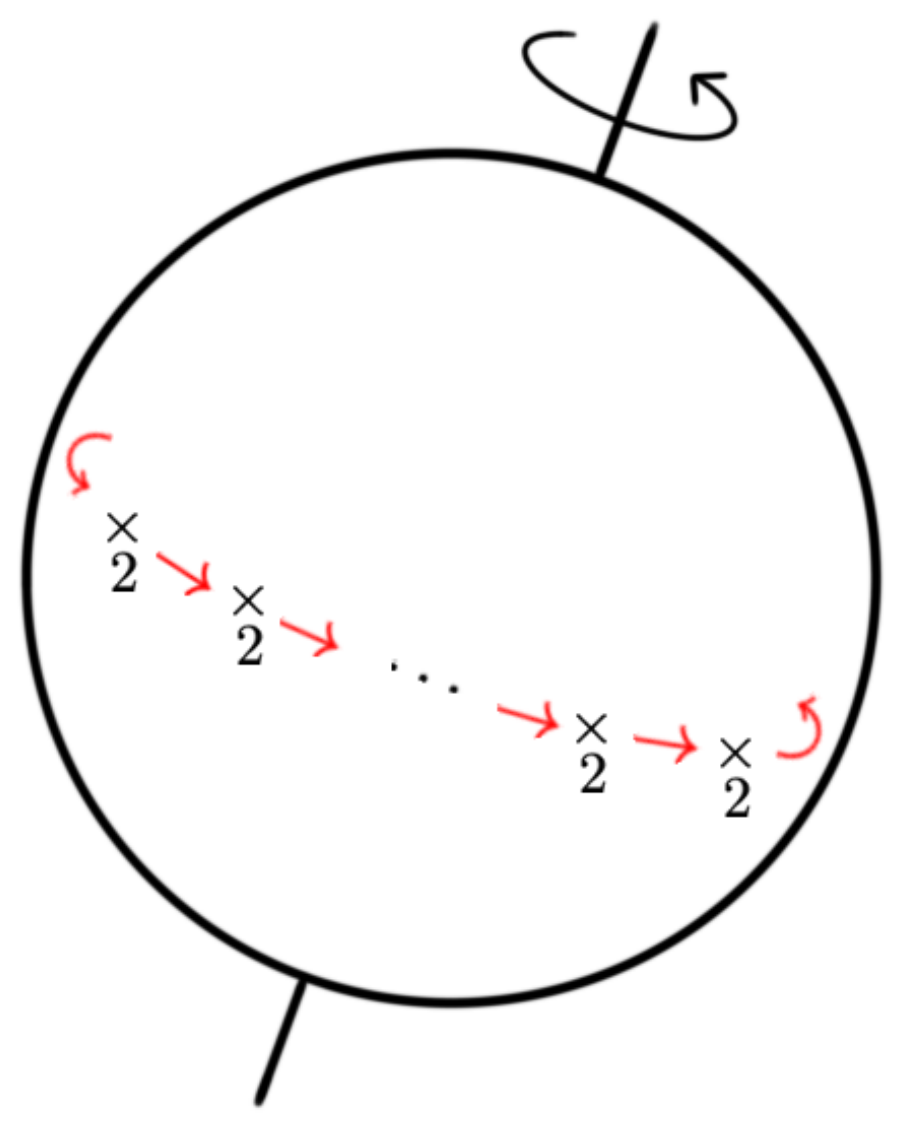}\\
				(iii) The middle orbit consists of $2g+1$ cone points.
			\end{minipage}
			\caption{\label{fig:rot}}
		\end{figure}
		
	Take any commuting pair of a periodic diffeomorphism $f$ and a hyperelliptic involution $I$ on $\Sigma_{g}$.
	Let $\bar{f}$ be a periodic diffeomorphism of $\Sigma_{g} / \langle I \rangle$ induced by $f$.
	Let $n=\opn{ord}(f)$ and $\bar{n}=\opn{ord}(\bar{f})$. We have the following simple observation.

		\begin{lem}
		\label{lem:1}
			\begin{enumerate}
				\item We have $n = 2\bar{n}$  if and only if $f^{\frac{n}{2}} = I$.
				\item We have $n = \bar{n}$  if and only if $f^{\frac{n}{2}} \neq I$.
			\end{enumerate}
		\end{lem}
		
		\begin{proof}
			If $n = 2\bar{n}$, then we have $\bar{f}^{\frac{n}{2}} = \opn{id}_{\Sigma_{g}/\langle I \rangle}$.
			Note that the lifts of $\opn{id}_{\Sigma_{g} / \langle I \rangle}$ to $\Sigma_{g}$ are $I$ and $\opn{id}_{\Sigma_{g}}$.
			Therefore, since $n \neq \bar{n}$ by assumption, it follows that $f^{\frac{n}{2}} = I$.
			The converse is obvious.
			Here (2) follows from (1) by taking the contraposition.
		\end{proof}

	Let $S^2(\lambda_1, \ldots, \lambda_s)$ denote a sphere with $s$ cone points with indices $\lambda_1, \lambda_{2}, \ldots, \lambda_{s}$.
	Since $\Sigma_{g} / \langle I \rangle$ is diffeomorphic to $S^2(2,2,\ldots,2)$ with $2g+2$ cone points, it follows that $\bar{f}$ is a rotation which fixes two points on $S^2(2,2,\ldots,2)$.
	Thus, the cardinality of $\opn{Fix}( \bar{f} ) \cap B_I$ is equal to either of $0$, $1$ or $2$, where $B_{I}$ is the set of cone points of $S^2(2,2,\ldots,2)$.
	If $|\opn{Fix}( \bar{f} ) \cap B_I| = i$ for $i=0$, $1$ and $2$, then the set of other $2g+2-i$ cone points of $S^2(2,2,\ldots,2)$ are divided into $\bar{f}$-orbits.
	Thus, $\bar{f}$ is a power of a periodic diffeomorphism $\bar{h}$ of $S^2(2,2,\ldots,2)$ whose period is $2g+2-i$ (see Figure \ref{fig:rot}).
	Let $h$ be a lift of $\bar{h}$ to $\Sigma_{g}$.
	Here $h$ commutes with $I$, and $\langle f, I \rangle$ is a subgroup of $\langle h, I \rangle$.
	Therefore, by replacing $f$ with $h$, it is sufficient to consider the following three cases in order to prove Theorem \ref{thm:Ishizaka}:

		\begin{enumerate}
		\renewcommand{\labelenumi}{(\roman{enumi})}
			\item $\bar{f}$ is the $\frac{2\pi}{2g}$-rotation of $S^2(2,2,\ldots,2)$ and $\Sigma_{g} / G \cong S^2 (4g, 4g , 2)$,
			\item $\bar{f}$ is the $\frac{2\pi}{2g+1}$-rotation of $S^2(2,2,\ldots,2)$ and $\Sigma_{g} / G \cong S^2 (4g + 2, 2g + 1 , 2)$,
			\item $\bar{f}$ is the $\frac{2\pi}{2g+2}$-rotation of $S^2(2,2,\ldots,2)$ and $\Sigma_{g} / G \cong S^2 (2g + 2, 2g + 2 , 2)$.
		\end{enumerate}

	First, we consider the cases (i) and (ii).
	Since $G$ is cyclic in these cases as we will see, it is easy to compute the total valency of $f$. 		
	
		\begin{lem}
		\label{lem:2}
			\begin{enumerate}
				\item In the cases (i) and (ii), the group $G$ is cyclic.
				\item In the case (i), the total valency of $f$ is equal to the total valency of a power of $f_{2}$.
				\item In the case (ii), the total valency of $f$ is equal to the total valency of a power of $f_{1}$.
            		\end{enumerate}
		\end{lem}
		
		\begin{proof}
			Let us prove (1).
			In other words, we show that $f^{\frac{n}{2}} = I$ in the cases (i) and (ii).
			Assume that we have $f^{\frac{n}{2}} \neq I$ in the case (i).
			Since the isotropy group of any point of $\Sigma_{g}$ is cyclic, if $\Sigma_{g} /G \cong S^2 (4g, 4g, 2)$, then $G$ must have an element of order at least $4g$.
			However, since $n= \bar{n} =2g$ by Lemma \ref{lem:1}, we have $G \cong \mathbb{Z}/(2g)\mathbb{Z} \oplus \mathbb{Z}/2\mathbb{Z}$ and it is impossible.
			Similarly, we can show that $G$ is cyclic in the case (ii).

			Let us prove (2).
			Since $G$ is cyclic by (1), the total valency of $f$ is of the form $[\,g,4g\,;\, \frac{\theta_1}{4g}+ \frac{\theta_2}{4g}+ \frac{1}{2} \,]$ for some integers $\theta_{1}$, $\theta_{2}$.
			By Nielsen's condition (Proposition \ref{prop:Nielsen}), we have that $\frac{\theta_{1}}{4g}+ \frac{\theta_{2}}{4g}+ \frac{1}{2} = m$ for some integer $m$.
			Since $\theta_{1}$ is an odd number, we have
				\[
					\theta_{2} = 4mg - 2g - \theta_{1} \equiv (2g-1) \theta_{1} \mod 4g.
				\]
			By Lemma \ref{lem:123}, the total valency of $f$ is equal to the total valency of the $\theta_{1}$-th power of $f_{2}$.
			
			We can prove (3) in a way similar to the proof of (2):
			Since $G$ is cyclic by (1), the total valency of $f$ is of the form $[\,g,4g+2\,;\, \frac{\theta_1}{4g+2}+ \frac{\theta_2}{2g+1}+ \frac{1}{2} \,]$.
			By using Nielsen's condition as above, we can show that $\theta_{2} \equiv g\theta_{1} \mod 2g+1$.
			By Lemma \ref{lem:123}, the total valency of $f$ is equal to that of the $\theta_{1}$-th power of $f_{1}$.
		\end{proof}
		
		\begin{cor}
		　 The map $f_i^{n/2}$ is the unique hyperelliptic involution which commutes with $f_{i}$ for $i=1$ and $2$.
		\end{cor}
	
	Theorem \ref{thm:vis} for the cases (i) and (ii) follows from Lemma \ref{lem:2} and Theorem \ref{thm:Nielsen}.
			
	Let us consider the case (iii).
	In this case, $G$ is a direct sum of two cyclic groups generated by $f$ and $I$ by Lemma \ref{lem:1}.
	We can compute the total valency of $f$ as follows.
	
		\begin{lem}
		\label{lem:3k}
			In the case (iii), we have the following.
				\begin{enumerate}
					\item We have $n=2g+2$.
					\item The map $\Sigma_{g} / F \to \Sigma_{g} / G$ is branched at two branch points of indices $2g+2$ and $2$, respectively.
						In particular, we have $\Sigma_{g} / F \cong S^2(2g+2, 2g+2, g+1)$.  
    					\item The total valency of $f$ is equal to the total valency of either $f_3^k$ or $I f_3^{k}$ for some integer $k$.
				 \end{enumerate}
		\end{lem}
			
		\begin{proof}
			Let us prove (1).
			Since $\bar{n}=2g+2$, we have either $n=2g+2$ or $n=4g+4$.
			Assume that $n = 4g+4$.
			Since $G$ is cyclic by Lemma \ref{lem:1}, the total valency of $f$ is of the form $[\, g, 4g+4 \,;\, \frac{\theta_{1}}{2g+2}+ \frac{\theta_{2}}{2g+2}+ \frac{1}{2} \,]$ for some integers $\theta_{1}$, $\theta_{2}$.
			It follows that $f^{2g+2} = id_{\Sigma_{g}}$, which contradicts with $n=4g+4$.
			Therefore, by contradiction, we have $n = 2g+2$.

			Let us prove (2). By Nielsen's theorem  (see Remark \ref{rem:inv}), the double branched covering $\Sigma_{g}/F \to \Sigma_{g}/G$ is branched at an even number of points.
			Since $\Sigma_{g} / F \to \Sigma_{g} / G$ can be branched only at branch points of $\Sigma_{g} \to \Sigma_{g} / G$, which consists of $3$ points, it follows that $\Sigma_{g} / F \to \Sigma_{g} / G$ is branched at two points.
			Since the inverse image in $\Sigma_{g}$ of the branch point of index $2$ on $\Sigma_{g} / G$ is a principal orbit of $f$, we see that $\Sigma_{g} / F \to \Sigma_{g} / G$ is branched at the branch point of $\Sigma_{g} \to \Sigma_{g} / G$ of branch index $2$.
			It follows that $\Sigma_{g} / F \to \Sigma_{g} / G$ is branched at two branch points of indices $2g+2$ and $2$, respectively.
			Thus, $\Sigma_{g} / F$ is diffeomorphic to $S^2(2g+2, 2g+2, g+1)$ as an orbifold. 
				
			Let us prove (3).
			Let $h=If$ if $f$ interchanges two cone points of index $2g+2$, and $h=f$ otherwise.
			In both cases, $h$ fixes two cone points of index $2g+2$.
			Then, the total valency of $h$ has the form $[\, g, 2g+2 \,;\, \frac{\theta_1}{2g+2} + \frac{\theta_2}{2g+2} + \frac{\theta_3}{g+1} \,]$ for some integers $\theta_1$, $\theta_{2}$ and $\theta_{3}$.
			Here, since $I$ interchanges two cone points of index $2g+2$, we have $\theta_1=\theta_2$.
			Set $k=\theta_1=\theta_{2}$.
			By Proposition \ref{prop:Nielsen}, we have that $\frac{k}{2g+2} + \frac{k}{2g+2} + \frac{\theta_3}{g+1}$ is an integer.
			It follows that $\theta_3 \equiv -k \equiv kg \mod g+1$.
			By Lemma \ref{lem:123}, the total valency of $h$ is equal to the total valency of the $k$-th power of $f_{3}$.
		\end{proof}
			
		\begin{lem}
		\label{lem:I}
			Any hyperelliptic involution $I'$ which commutes with $f_3$ is conjugate to $I$ in Figure \ref{fig:phi3} by a conjugacy which commutes with $f_{3}$.
		\end{lem}
			
		\begin{proof}
		    Let $G=G_{3} = \langle f_{3}, I \rangle$.
		    Let $G' = \langle f_{3}, I'\rangle$.
		    By Lemma \ref{lem:3k}-(2), we have $\Sigma_{g}/G \cong \Sigma_{g}/G' \cong S^{2}(2g+2,2g+2,g+1)$. 
		    Among two cone points of $\Sigma_{g}/G$ (resp.\ $\Sigma_{g}/G'$) of index $2g+2$, let $y$ (resp.\ $y'$) be the point in the branched locus of $\Sigma_{g} / F \to \Sigma_{g} / G$ (resp.\ $\Sigma_{g} / F \to \Sigma_{g} / G'$).
		    Then there exists a diffeomorphism $k_{G} : \Sigma_{g}/G \to \Sigma_{g}/G'$ such that $k_{G}(y)=y'$. 
			We can lift $k_{G}$ to a diffeomorphism $k_{F} : \Sigma_{g}/F \to \Sigma_{g}/F$.
			Since $k_{F}$ preserves the valency of each multiple orbit of $f$, we can lift $k_{F}$ to a diffeomorphism $k : \Sigma_{g} \to \Sigma_{g}$.
			By construction, we have $k_{F} \circ \bar{I} = \bar{I}' \circ k_{F}$, where $\bar{I}$, $\bar{I}'$ are diffeomorphisms of $\Sigma_{g}/F$ induced by $I$, $I'$, respectively.
			Then we have $k \circ f = f \circ k$ and $k \circ I = I' \circ k$, which concludes the proof.
		\end{proof}
				
	Theorem \ref{thm:vis} for the case (iii) follows from Lemmas \ref{lem:3k} and \ref{lem:I}.
			
		\begin{rem}
			We can see that $If_3$ is conjugate to $f_3^{2g+1}$ by computing the total valency of $If_{3}$.
			In particular, if $k$ is odd, then $If_3^k$ is conjugate to $f_3^{(2g+1)k}$.
			It means that we do not need to include $If_3^k$ in (iv) of Theorem \ref{thm:Ishizaka1} for odd $k$.
			Similarly we can see that $If_{3}^2$ is not conjugate to $f_{3}^{k}$ for any $k$ by using the total valency.
		\end{rem}

\section{Dehn twist presentations in the pointed mapping class group}
\label{sec:3}

	In this section, we will obtain right-handed Dehn twist presentations of the diffeomorphisms $f_{1}$, $f_{2}$ and $f_{3}$ in the pointed mapping class group to prove Theorem \ref{thm:Dehn} (see Figures \ref{fig:phi1}, \ref{fig:phi2} and \ref{fig:phi3} for the diffeomorphisms $f_{1}$, $f_{2}$ and $f_{3}$, respectively).
	For the purpose, we use the standard $\opn{Mod}(\Sigma_{g}, x)$-action on $\pi_{1}(\Sigma_{g},x)$, which is well known to be faithful (see, e.g. \cite[Chapter 8]{FM}).

		\begin{figure}[ht]
			\centering
			\begin{tabular}{ccc}
			\begin{minipage}{0.33\textwidth}
				\centering
				\captionsetup{width=0.85\linewidth}
				\includegraphics[width=1\textwidth ]{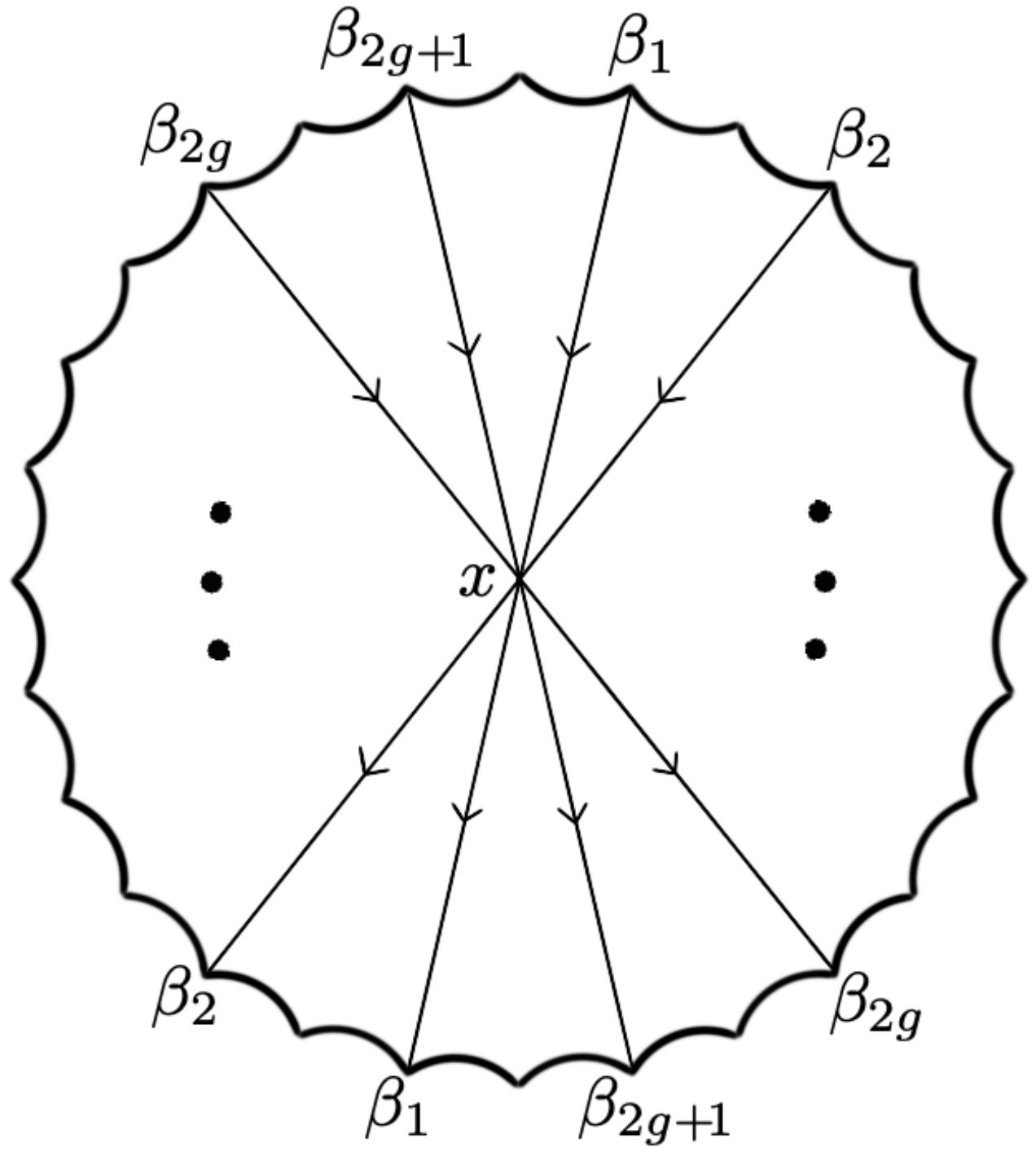}
				\caption{A system of generators of $\pi_1(\Sigma_{g})$ for $\Sigma_{g}$ in Fig.\ \ref{fig:phi1}}
				\label{fig:phi1-3}
			\end{minipage}\hfill
			\begin{minipage}{0.33\textwidth}
				\centering
				\captionsetup{width=0.85\linewidth}
				\includegraphics[width=1\textwidth ]{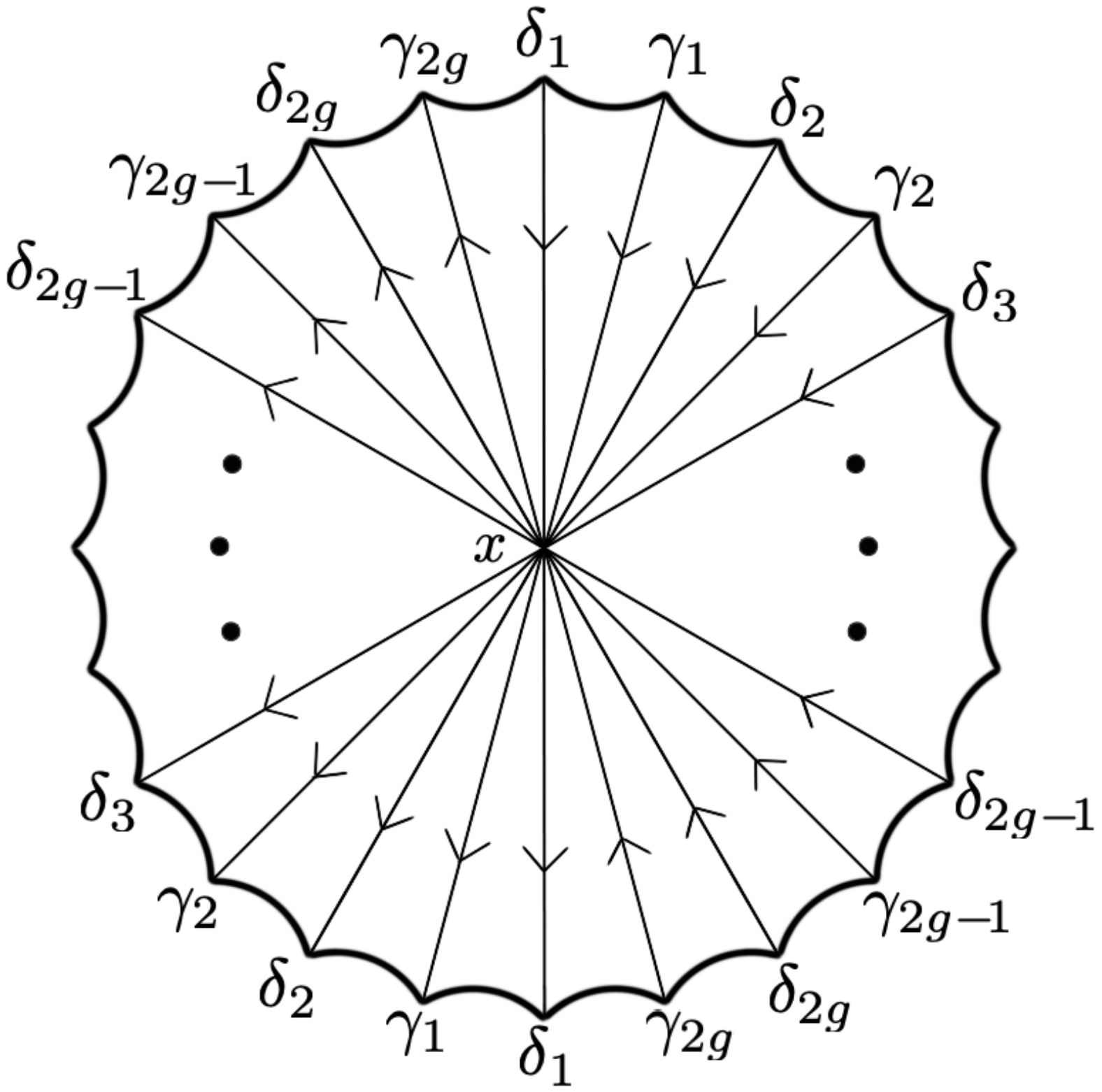}
				\caption{A system of generators of $\pi_1(\Sigma_{g})$ for $\Sigma_{g}$ in Fig.\ \ref{fig:phi2}}
				\label{fig:phi2-3}
			\end{minipage}\hfill
			\begin{minipage}{0.33\textwidth}
				\centering
				\captionsetup{width=0.85\linewidth}
				\includegraphics[width=1\textwidth ]{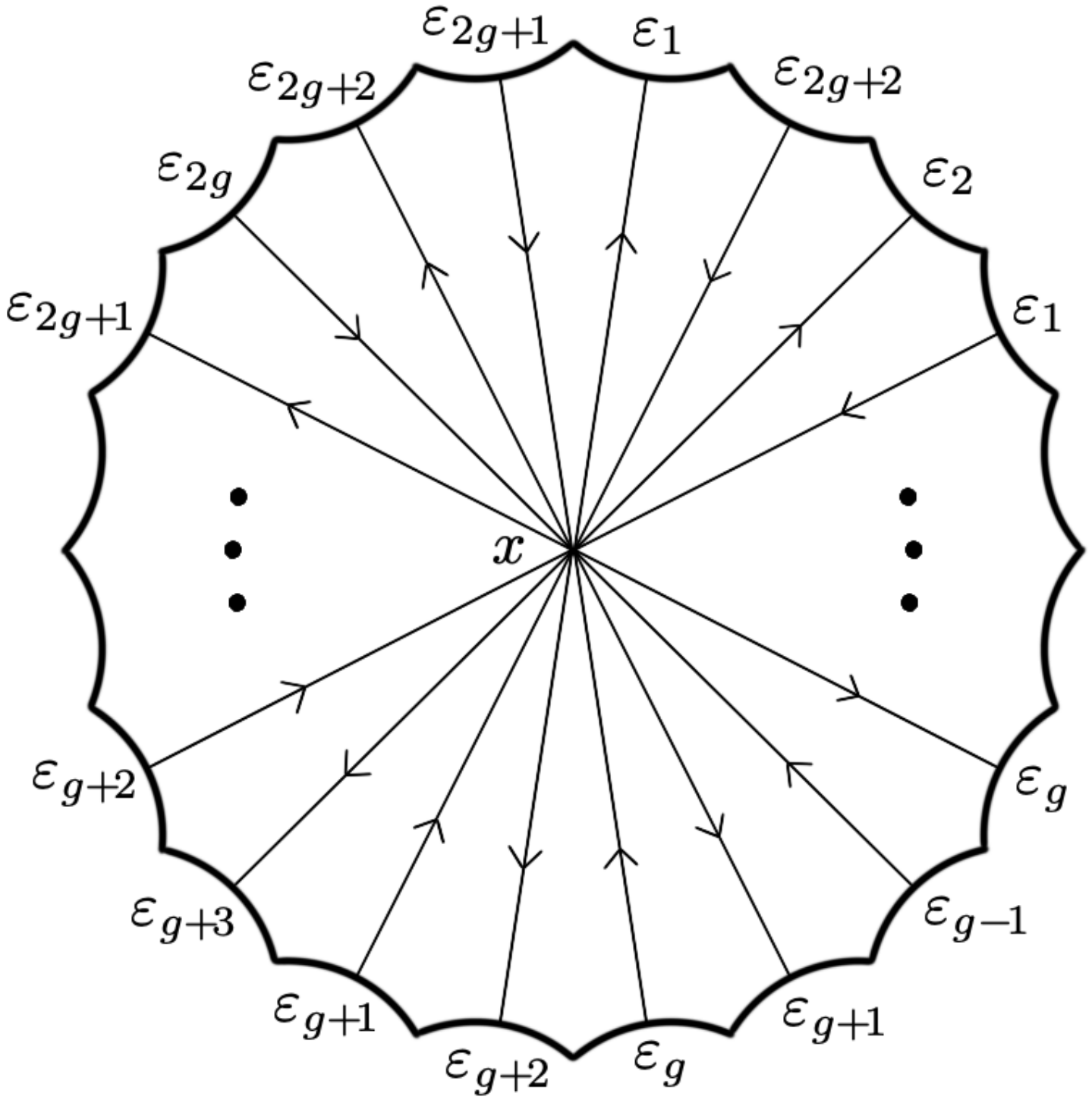}
				\caption{A system of generators of $\pi_1(\Sigma_{g})$ for $\Sigma_{g}$ in Fig.\ \ref{fig:phi3}}
				\label{fig:phi3-3}
			\end{minipage}
			\end{tabular}
		\end{figure}
	
		\begin{prop}\label{prop:f1}
			Take simple closed curves $A_{1}, A_{2}, \ldots, A_{2g+1}$ on $\Sigma_{g}$ as in Figure \ref{fig:phi1-2}.
			A product $A_{1} A_{2} \cdots A_{2g}$ of Dehn twists and the pointed mapping class of $f_1$ induce the same automorphism of $\pi_{1}(\Sigma_{g},x)$.
		\end{prop}
		
		\begin{proof}
			Take $\beta_1, \beta_2, \ldots, \beta_{2g+1} \in \pi_1(\Sigma_{g},x)$ as in Figure \ref{fig:phi1-3}.
			It is easy to see that $\pi_1(\Sigma_{g},x)$ is generated by these $\beta_1, \beta_2, \ldots, \beta_{2g+1}$.
			Since we have 
				\[
					\beta_{2g+1} \displaystyle\prod^{g-1}_{i=0} \beta_{2g-2i}^{-1} \beta_{2g-2i-1} = 1,
				\]
			it follows that $\pi_1(\Sigma_{g},x)$ is generated by $\beta_1, \ldots, \beta_{2g}$.
			Here clearly we have $f_{1}(\beta_{i}) = \beta_{i+1}$ for $i=1, \ldots, 2g$.
			As shown in Figure \ref{fig:phi1ac}, we can see that $A_i\ ( i = 1,2, \cdots , 2g+1)$ acts on $\pi_1(\Sigma_{g},x)$ by 
				\begin{align*}
					A_{i+1}(\beta_i) &= \beta_{i+1} \quad\quad (1 \leq i \leq 2g),\\
					A_j(\beta_i) &= \beta_i \quad\quad ( j \leq i-1 \ \text{or}\   j \geq i+2).
				\end{align*}
			Thus, we have
				\[
					A_{1} A_{2} \cdots A_{2g} (\beta_i)=A_{1} A_{2} \cdots A_{i} A_{i+1} (\beta_i)   
						= A_{1} A_{2} \cdots A_{i}(\beta_{i+1})
						= \beta_{i+1}
				\]
			for $i = 1, \ldots, 2g$.
			Since the action of $f_1$ on $\pi_{1}(\Sigma_{g},x)$ coincides with that of the Dehn twist presentation $A_{1} A_{2} \cdots A_{2g}$, the proposition follows.
		\end{proof}

		\begin{figure}[ht]
			\centering
			\begin{minipage}{0.33\textwidth}
				\centering
				\includegraphics[width=1\textwidth ]{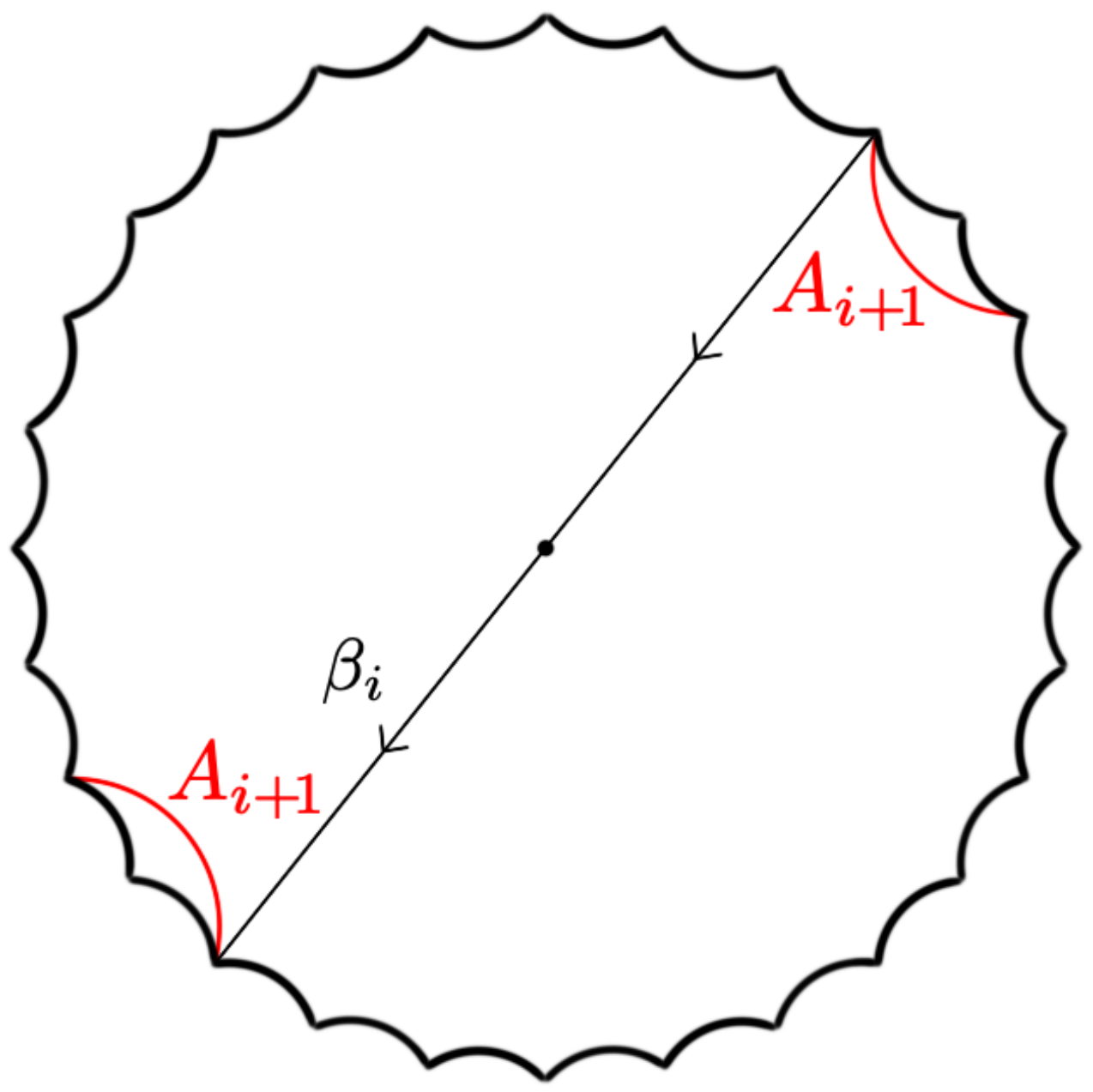}
			\end{minipage}\hfill
			\begin{minipage}{0.33\textwidth}
				\centering
				\includegraphics[width=1\textwidth ]{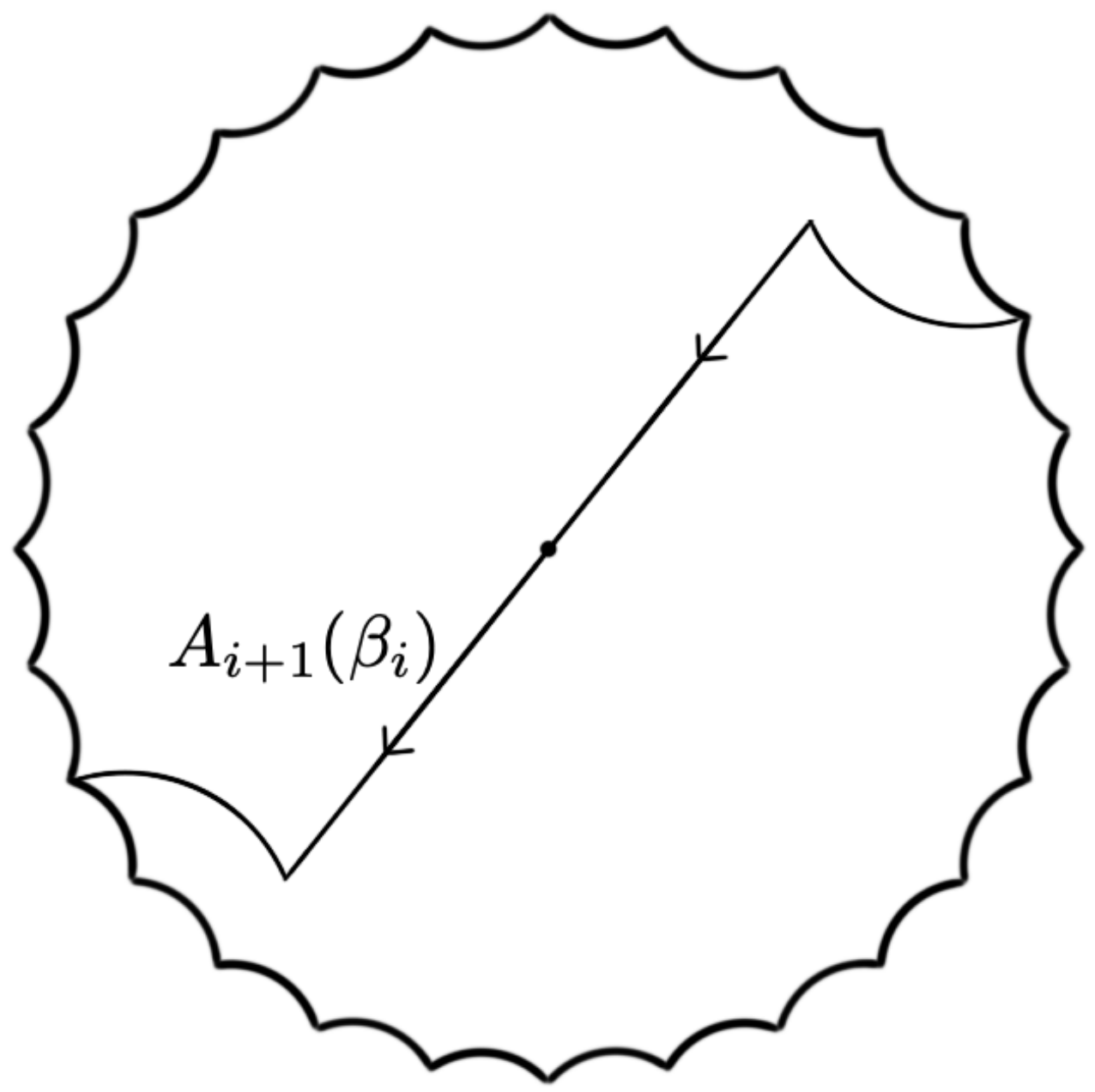}
			\end{minipage}\hfill
			\begin{minipage}{0.33\textwidth}
				\centering
				\includegraphics[width=1\textwidth ]{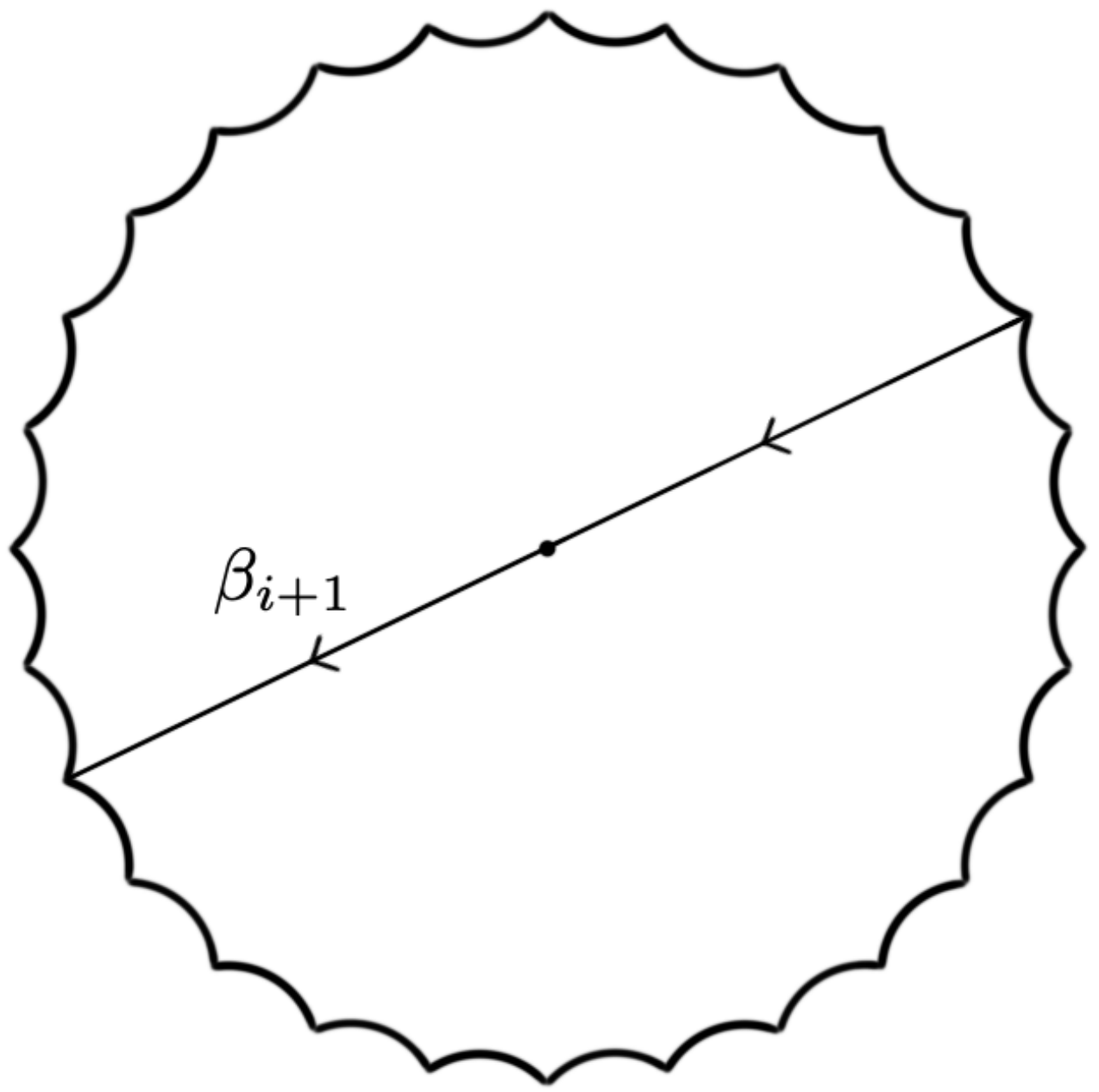}
			\end{minipage}
			\caption{}
			\label{fig:phi1ac}
		\end{figure}

		\begin{prop}\label{prop:f2}
			Take simple closed curves $B_1, B_2, \ldots, B_{2g}$ on $\Sigma_{g}$ as in Figure \ref{fig:phi2-2}.
			A product $B_{1} B_{2g} B_{2g-1} \cdots B_{1}$ of Dehn twists and the pointed mapping class of $f_2$ induce the same automorphism of $\pi_1(\Sigma_{g}, x)$.
		\end{prop}
		
		\begin{proof}
			Take $\gamma_1, \gamma_2, \ldots, \gamma_{2g}, \delta_1, \delta_2, \ldots, \delta_{2g} \in \pi_1(\Sigma_{g},x)$ as shown in Figure \ref{fig:phi2-3}.
			It is easy to see that $\pi_1(\Sigma_{g},x)$ is generated by these $\gamma_1, \gamma_2, \ldots, \gamma_{2g}, \delta_1$.
			Since we have 
				\[
					\delta_1 = \displaystyle\prod^{g-1}_{i=0} \gamma_{2g-2i}^{-1} \gamma_{2g-2i-1},
				\]
			it follows that $\pi_1(\Sigma_{g},x)$ is generated by $\gamma_1, \ldots, \gamma_{2g}$.
			Here clearly we have
				\[ 
					f_{2}(\gamma_{i}) = \gamma_{i+1} \ (i=1, \ldots, 2g-1)\  \text{and}\  f_{2}(\gamma_{2g})=\gamma_1^{-1}.
				\]
			We can see that $B_i\ ( i = 1,2, \ldots, 2g)$ acts on $\pi_1(\Sigma_{g},x)$ by 
				\begin{align*}
					B_i(\delta_i) &= \gamma_{i} \quad\quad (1 \leq i \leq 2g),\\
					B_i(\gamma_i) &= \delta_{i+1} \quad\quad ( 1 \leq i \leq 2g-1),\\
					B_{2g}(\gamma_{2g}) &=\delta_1^{-1},\\
					B_j(\gamma_i) &=\gamma_i \quad\quad (i\neq j).
				\end{align*}
			Thus, we have
				\begin{align*}
					B_{1} B_{2g} B_{2g-1} \cdots B_i \cdots B_{1} (\gamma_i)
						&= B_{1} B_{2g} B_{2g-1} \cdots B_{i}(\gamma_{i})\\
						&= B_{1} B_{2g} B_{2g-1} \cdots B_{i+1}(\delta_{i+1})\\
						&= B_{1} B_{2g} B_{2g-1} \cdots B_{i+2} (\gamma_{i+1}) = \gamma_{i+1},
				\end{align*}
			for $i=1,\ldots,2g-1$ and 
				\[
					B_{1} B_{2g} B_{2g-1} \cdots B_{1} (\gamma_{2g})
						= B_{1} B_{2g}(\gamma_{2g})
						= B_{1}  (\delta_{1}^{-1}) = \gamma_{1}^{-1}.
				\]
			Therefore, since the action of $f_2$ on $\pi_{1}(\Sigma_{g},x)$ coincides with that of the Dehn twist presentation $B_{1} B_{2g} B_{2g-1} \cdots B_{1}$, the proposition follows.
		\end{proof}

		\begin{prop}\label{prop:f3}
			Take simple closed curves $D_1, D_2, \ldots, D_{2g+1}$ on $\Sigma_{g}$ as in Figure \ref{fig:phi3-2}.
			A product $D_{1} D_{2} \cdots D_{2g+1}$ of Dehn twists and the pointed mapping class of $f_2$ induce the same automorphism of $\pi_1(\Sigma_{g}, x)$
		\end{prop}
		
		\begin{proof}
			Take $\varepsilon_1, \varepsilon_2, \ldots, \varepsilon_{2g+1} \in \pi_1(\Sigma_{g},x)$ on $\Sigma_{g}$ as shown in Figure \ref{fig:phi3-3}.
						It is easy to see that $\pi_1(\Sigma_{g},x)$ is generated by these $\varepsilon_1, \varepsilon_2, \ldots, \varepsilon_{2g+2}$.
			Since we have 
				\[
					\displaystyle\prod^{g}_{i=0} \varepsilon_{2g+2-2i} = \displaystyle\prod^{g}_{i=0} \varepsilon_{2g+1-2i} = \displaystyle\prod^{2g+1}_{i=0} \varepsilon_{2g+2-i} = 1,
				\]
			it follows that $\pi_1(\Sigma_{g},x)$ is generated by $\varepsilon_1, \ldots, \varepsilon_{2g}$.
			Here clearly we have $f_{3}(\varepsilon_{i}) = \varepsilon_{i+1}$ for $i=1, \ldots, 2g$.
		We can see that $D_i \ ( i = 1,2, \ldots, 2g+1)$ acts on $\pi_1(\Sigma_{g},x)$ by
				\begin{align*}
					D_i(\varepsilon_{i+1}) &=\varepsilon _{i+1} \varepsilon_i^{-1} \quad\quad ( 1 \leq i \leq 2g+1 ),\\
					D_{i+1}(\varepsilon_i) &= \varepsilon_{i+1}\varepsilon_i \quad\quad ( 1 \leq i \leq 2g+1 ),\\
					D_j(\varepsilon_i) &= \varepsilon_i \quad\quad (\mid i - j \mid \neq 1).
				\end{align*}
			Thus, we have
				\begin{align*}
					D_{1} D_{2} \cdots D_i D_{i+1} \cdots D_{2g+1} (\varepsilon_i)
						&= D_{1} D_{2} \cdots D_{i}(\varepsilon_{i+1}\varepsilon_i)\\
						&= D_{1} D_{2} \cdots D_{i-1} (\varepsilon_{i+1} \varepsilon_i ^{-1} \varepsilon_i)\\
						&= D_{1} D_{2} \cdots D_{i-1} (\varepsilon_{i+1}) = \varepsilon_{i+1},
				\end{align*}
			for $i = 1, \cdots 2g$.
			Since the action of $f_3$ on $\pi_{1}(\Sigma_{g},x)$ coincides with that of the Dehn twist presentation $D_{1} D_{2} \cdots D_{2g+1}$, the proposition follows.
		\end{proof}
		
Theorem \ref{thm:Dehn} is a combination of Propositions \ref{prop:f1}, \ref{prop:f2} and \ref{prop:f3}.

\end{document}